\documentclass[12pt,a4paper]{amsart}
\usepackage[english]{babel}
\usepackage[applemac]{inputenc}
\usepackage[T1]{fontenc}
\usepackage{palatino}
\usepackage{amsmath}
\usepackage{amssymb}
\usepackage{amsthm}
\usepackage{amsfonts}
\usepackage{graphicx}
\usepackage[colorlinks = true, citecolor = black]{hyperref}
\author{Tuomas Orponen}\thanks{The research was supported by the Finnish Centre of Excellence in Analysis and Dynamics Research}
\title{On the Distance Sets of Self-Similar Sets}
\address{Department of Mathematics and Statistics, University of Helsinki, P.O.B. 68, FI-00014 University of Helsinki, Finland}
\email{tuomas.orponen@helsinki.fi} 

\newcommand{\R}{\mathbb{R}}
\newcommand{\N}{\mathbb{N}}

\newcommand{\calB}{\mathcal{B}}
\newcommand{\calD}{\mathcal{D}}
\newcommand{\calC}{\mathcal{C}}

\newcommand{\calH}{\mathcal{H}}

\newcommand{\calG}{\mathcal{G}}

\newcommand{\card}{\operatorname{card}}
\newcommand{\Id}{\operatorname{Id}}

\numberwithin{equation}{section}

\theoremstyle{plain}
\newtheorem{thm}[equation]{Theorem}
\newtheorem{lemma}[equation]{Lemma}

\newtheorem{cor}[equation]{Corollary}
\newtheorem{proposition}[equation]{Proposition}

\newtheorem{conjecture}[equation]{Conjecture}

\theoremstyle{definition}
\newtheorem{definition}[equation]{Definition}

\theoremstyle{remark}
\newtheorem{remark}[equation]{Remark}

\begin{document}

\begin{abstract} We show that if $K$ is a self-similar set in the plane with positive length, then the distance set of $K$ has Hausdorff dimension one.
\end{abstract}

\maketitle

\section{Introduction}

The \emph{distance set problem} usually refers to two different but closely related questions, one in incidence geometry and the other in geometric measure theory. The incidence geometric formulation, due to P. Erd{\H o}s \cite{E} from 1946, asks to determine the least number of distinct distances spanned by a set of $n \in \N$ points in the plane. The conjectured bound is $\geq cn(\log n)^{-1/2}$ distances. A recent proof by L. Guth and N.H. Katz \cite{GK} comes very close by extracting $\geq cn(\log n)^{-1}$ distances.

The present paper deals with a special case of the geometric measure theoretic version of the distance set problem, formulated by K. Falconer \cite{Fa} in 1985. In \cite{Fa}, Falconer proved that if $B \subset \R^{d}$ is an analytic set and $t < \dim B - d/2 + 1/2$, then the \emph{distance set} $D(B) = \{|x - y| : x,y \in B\}$ has positive $t$-dimensional Hausdorff measure. In particular, we have $\dim D(B) \geq \dim B - d/2 + 1/2$. Here and below, $\dim$ will always refer to Hausdorff dimension. For $d \geq 2$, Falconer also constructed examples of compact sets $K \subset \R^{d}$ with $\dim K = s$ and $\dim D(K) \leq 2s/d$, for any $0 \leq s \leq d$. The question on the sharpness of these bounds and examples became known as Falconer's distance set problem. The constructions in \cite{Fa} show that the bounds given by the following conjecture would be optimal.
\begin{conjecture}[Distance set conjecture]\label{distanceConjecture} Let $B \subset \R^{d}$ be an analytic set. If $\dim B \geq d/2$, then $\dim D(B) = 1$. If $\dim B > d/2$, then $D(B)$ has positive length.
\end{conjecture}

Since 1985, the conjecture and its variations have been studied by many authors, including J. Bourgain \cite{Bo}, \cite{Bo2}, T. Wolff \cite{Wo}, B. Erdo{\u g}an \cite{Er}, P. Mattila \cite{Ma1}, Mattila and P. Sj\"olin \cite{MS}, Katz and T. Tao \cite{KT}, T. Mitsis \cite{Mi}, Y. Peres and W. Schlag \cite{PSc}, S. Hofmann and A. Iosevich \cite{HI}, Iosevich and I. {\L}aba \cite{IL}, \cite{IL2}, {\L}aba and S. Konyagin \cite{LG} and Iosevich, M. Mourgoglou and K. Taylor \cite{IMT}. Readers unfamiliar with the problem may wish to consult the introduction in \cite{Er}, which contains a brief summary of most developments up until 2004.

In this paper, we consider the planar case of Conjecture \ref{distanceConjecture} for self-similar sets. Before stating the main result, let us quickly review the best known estimates on the dimension and measure of distance sets of general (analytic) sets $B \subset \R^{2}$. The main result in Wolff's paper \cite{Wo} states that if $\dim B > 3/4$, then $D(B)$ has positive length. In \cite{Bo2}, Bourgain shows that if $\dim B \geq 1$, then $\dim D(B) \geq 1/2 + \varepsilon$ for some absolute constant $\varepsilon > 0$, smaller than $1/2$. We prove:
\begin{thm}\label{main} Let $K \subset \R^{2}$ be a self-similar set with $\calH^{1}(K) > 0$. Then $\dim D(K) = 1$.
\end{thm}

Our method does not yield positive length for $D(K)$ under any assumption, nor can we show that $\dim K \geq 1$ alone would imply $\dim D(K) = 1$. Thus, Conjecture \ref{distanceConjecture} remains open even in the self-similar case. Theorem \ref{main} imposes no 'separation conditions' on the self-similar sets under consideration, but readers familiar with the terminology should note that, according to a result of A. Schief \cite{Sc}, the open set condition is a consequence of the assumption $\calH^{1}(K) > 0$ in case the similarity dimension of the function system generating $K$ equals one. 

The proof of Theorem \ref{main} splits into two completely disjoint parts, according to whether or not the self-similar set $K$ contains an irrational rotation. The proof in the presence of irrational rotations will be very short, but only so because we have the following deep result by M. Hochman and P. Shmerkin at our disposal.
\begin{thm}[Adapted from Corollary 1.7 in \cite{HS}]\label{HSthm} Let $K$ be a self-similar set in the plane satisfying the strong separation condition. Assume that at least one of the similitudes generating $K$ contains an irrational rotation. Then, for any $C^{1}$-mapping $g \colon K \to \R$ without singular points, we have
\begin{displaymath} \dim g(K) = \min\{1,\dim K\}. \end{displaymath}  
\end{thm}

The assumption on $g$ simply means that $g$ is continuously differentiable in a neighborhood of $K$ and $\nabla g(x) \neq 0$ for $x \in K$. From this theorem we derive:
\begin{cor}\label{irrat} Let $K$ be a self-similar set in the plane. Assume that at least one of the similitudes generating $K$ contains an irrational rotation. Then, we may find a point $x \in K$ such that $\dim D_{x}(K) = \min\{1,\dim K\}$, where $D_{x}(K)$ is the \emph{pinned distance set}
\begin{displaymath} D_{x}(K) = \{|x - y| : y \in K\}. \end{displaymath}
\end{cor} 

It may seem surprising at first sight that Corollary \ref{irrat} can be deduced from Theorem \ref{HSthm}, which, unlike Corollary \ref{irrat}, imposes the strong separation condition on the set $K$. The explanation is simple: an easy argument given in \S3 shows that all separation assumptions can also be dropped from Theorem \ref{HSthm}. We point out that the conclusion in Corollary \ref{irrat} is notably stronger than the one in Theorem \ref{main}. In particular, Corollary \ref{irrat} implies the first part of Conjecture \ref{distanceConjecture} as such. Unfortunately, the presence of irrational rotations is essential for Theorem \ref{HSthm}, so there is no hope to apply Hochman and Shmerkin's result directly in the case with only rational rotations. Consequently, most of what follows will be devoted to the proof of Theorem \ref{main} in the rational case.
\section{Acknowledgements}

I am grateful to my advisor Prof. Pertti Mattila for suggesting the problem. I would also like thank an anonymous referee for making numerous wise remarks on the paper's exposition.  

\section{Definitions, and the case with irrational rotations}

\begin{definition}[Similitudes and types] A mapping $\psi \colon \R^{2} \to \R^{2}$ is a \emph{contractive similitude on $\R^{2}$}, or simply \emph{similitude}, if there exists a \emph{contraction ratio} $r \in (0,1)$ such that $|\psi(x) - \psi(y)| = r|x - y|$ for all $x,y \in \R^{2}$. Then $\psi$ may be then written in the form $\psi(x) = rROx + w$, where
\begin{itemize}
\item[(i)] $R = \Id$, or $R$ is the \emph{reflection} $R(z) = \bar{z}$,
\item[(ii)] $O$ is, in complex notation, the \emph{rotation} $O(z) = e^{2\pi i \theta}z$ for some $\theta \in [0,1)$,
\item[(iii)] $w \in \R^{2}$ is a \emph{translation vector}.
\end{itemize}
If $\psi(x) = rROx + w$, we say that the \emph{type} of $\psi$ is $[O,R]$ and write $[\psi] = [O,R]$. If the angle $\theta \in [0,1)$ from the definition of $O$ is a rational number, we say that $[O,R]$ is a \emph{rational type}. The similitude $\psi$ \emph{contains no rotations or reflections}, if $[\psi] = [\Id,\Id]$.
\end{definition} 
Fix a collection $\{\psi_{1},\ldots,\psi_{q}\}$ of contractive similitudes on $\R^{2}$. A classical result of Hutchinson \cite{Hu} states that there exists a unique compact set $K \subset \R^{2}$ satisfying
\begin{displaymath} K = \bigcup_{j = 1}^{q} \psi_{j}(K). \end{displaymath}
The set $K$ is called \emph{the self-similar set generated by the similitudes $\{\psi_{1},\ldots,\psi_{q}\}$}. If not otherwise specified, the letter $K$ will always refer to the the self-similar set generated by the similitudes $\{\psi_{1},\ldots,\psi_{q}\}$, and we will assume that $\calH^{1}(K) > 0$. For the rest of the paper, let $B_{0} := B(0,1/2) := \{x \in \R^{2} : |x| \leq 1/2\}$, and assume that $\psi_{j}(B_{0}) \subset B_{0}$ for $1 \leq j \leq q$ (this can be achieved by scaling all the translation vectors by some common factor, which has no impact on $\dim D(K)$ or any of the other results below). 

\begin{definition}[Generation $n$ balls]\label{genN} Set $\calB_{0} = \{B_{0}\}$, and define the collections $\calB_{n}$ recursively by
\begin{displaymath} \calB_{n} := \{\psi_{j}(B) : B \in \calB_{n - 1},\: 1 \leq j \leq q\}. \end{displaymath} 
Note that the union of these \emph{generation $n$ balls of $K$ associated with $\{\psi_{1},\ldots,\psi_{q}\}$} forms a cover for the set $K$ for any $n \in \N$. If $B \in \calB_{n}$ with $n \geq 1$, there exists a -- not necessarily unique -- sequence of mappings $(\psi_{i_{1}},\ldots,\psi_{i_{n}})$ with $i_{j} \in \{1,\ldots,q\}$ such that $B = \psi_{i_{1}} \circ \ldots \circ \psi_{i_{n}}(B_{0})$. The composition $\psi_{i_{1}} \circ \ldots \circ \psi_{i_{n}}$ is again a similitude on $\R^{2}$ with contraction ratio $d(B)/d(B_{0}) = d(B)$, and we write $\psi_{B} := \psi_{i_{1}} \circ \ldots \circ \psi_{i_{n}}$; if the sequence $(\psi_{i_{1}},\ldots,\psi_{i_{n}})$ is not uniquely determined by $B$, we pick any admissible sequence in the definition of $\psi_{B}$. If $\tilde{\calB} \subset \bigcup_{n \in \N} \calB_{n}$ is any subcollection of generation $n$ balls, for various $n$ perhaps, then the self-similar set $\tilde{K}$ generated by the similitudes $\{\psi_{B} : B \in \tilde{\calB}\}$ is a subset of $K$; this subset may be proper even if $\tilde{\calB} = \bigcup_{n \in \N} \calB_{n}$. Note that the first generation balls $\tilde{\calB}_{1}$ of $\tilde{K}$ (associated with the natural similitudes $\{\psi_{B} : B \in \tilde{\calB}\}$) agree with the balls in $\tilde{\calB}$. 
\end{definition}

\begin{definition}[Very strong separation] We say that the self-similar set $K$ satisfies the \emph{very strong separation condition}, if the first generation balls of $K$ are disjoint. This degree of separation will be convenient to us, and it is stronger than the commonly used notion of 'strong separation', which merely requires that $\psi_{i}(K) \cap \psi_{j}(K) = \emptyset$ for $1 \leq i \neq j \leq q$ (the strong separation assumption in Theorem \ref{HSthm} refers to this mainstream definition). Under either of these conditions, it is well-known, see \cite{Hu}, that $\dim K$ equals the unique number $s \in (0,2]$ solving the equation
\begin{displaymath} \sum_{j = 1}^{q} r_{j}^{s} = 1. \end{displaymath} 
\end{definition}
The strong separation assumption is not needed in Theorem \ref{HSthm} because of

\begin{lemma}\label{mainL1} Let $K$ be a self-similar set in $\R^{2}$. Then, for any $\varepsilon > 0$, there exists a self-similar set $K_{\varepsilon} \subset K$ satisfying the very strong separation condition and with 
\begin{displaymath} \dim K_{\varepsilon} > \dim K - \varepsilon. \end{displaymath}
\end{lemma}

\begin{proof} Define the collections $\calB_{n}$ as above, with $B_{0} = B(0,1/2)$ and so on. For $n \in \N$, use the $5r$-covering lemma, see \cite[Lemma 7.3]{Ru}, to find a subcollection $\calD_{n} \subset \calB_{n}$ of disjoint balls with the property that
\begin{equation}\label{form1} K \subset \bigcup_{B \in \calB_{n}} B \subset \bigcup_{B \in \calD_{n}} 5B. \end{equation} 
For any $n \in \N$, the similitudes $\{\psi_{B} : B \in \calD_{n}\}$ generate a self-similar set $K_{n}$ contained in $K$ and satisfying the very strong separation condition. Since $d(B)$ is the contraction ratio of the similitude $\psi_{B}$, the dimension of $K_{n}$ equals the unique number $s_{n} \in [0,\dim R]$ satisfying
\begin{equation}\label{dimsum} \sum_{B \in \calD_{n}} d(B)^{s_{n}} = 1. \end{equation} 
On the other hand, the balls $5B$, $B \in \calD_{n}$, cover the original set $K$, so the very definition of Hausdorff dimension implies that
\begin{displaymath} 5^{s} \sum_{B \in \calD_{n}} d(B)^{s} = \sum_{B \in \calD_{n}} d(5B)^{s} \to \infty \end{displaymath} 
as $n \to \infty$ for any $s < \dim K$. This forces $s_{n} \to \dim K$ as $n \to \infty$. \end{proof}

\begin{remark} If one of the similitudes generating $K$, say $\psi_{j}$, contains an irrational rotation (that is, $[\psi]$ is not a rational type) then the similitudes generating $K_{\varepsilon}$ can always be chosen so that one of them contains an irrational rotation as well. Indeed, suppose that the similitudes in $\Psi_{n} := \{\psi_{B} : B \in \calD_{n}\}$, as defined in the previous lemma, only contain rational rotations for large enough $n \in \N$. Then, for $n \in \N$, single out any one of the similitudes $\psi_{B}$, $B \in \calD_{n}$, and consider the collection $\tilde{\Psi}_{n}$, where this one similitude has been replaced by $\psi_{B} \circ \psi_{j}$. Then $\psi_{B} \circ \psi_{j} \in \tilde{\Psi}_{n}$ contains an irrational rotation. Moreover, the self-similar set $\tilde{K}_{n}$ generated by $\tilde{\Psi}_{n}$ still satisfies the very strong separation condition, and we also have $\dim \tilde{K}_{n} \to \dim K$, since (at least if $\dim R > 0$) the effect of any single term on the sum on line \eqref{dimsum} becomes negligible for $n \in \N$ large enough.
\end{remark}

\begin{proof}[Proof of Corollary \ref{irrat}] First, let us use the previous lemma and remark to rid Theorem \ref{HSthm} of the strong separation assumption: if $K \subset \R^{2}$ is a self-similar set containing an irrational rotation and $\varepsilon > 0$, find a self-similar set $K_{\varepsilon} \subset K$, containing an irrational rotation, satisfying the very strong separation condition, and with $\dim K_{\varepsilon} > \dim K - \varepsilon$. Then, for any $C^{1}$-mapping $g$ on $K$ without singular points, we have
\begin{displaymath} \dim g(K) \geq \dim g(K_{\varepsilon}) = \min\{1,\dim K_{\varepsilon}\} \geq \min\{1,\dim K - \varepsilon\}, \end{displaymath} 
which means that $\dim g(K) = \min\{1,\dim K\}$. To prove Corollary \ref{irrat}, we may assume that $\dim K > 0$. Then, for large enough $n \in \N$, we can choose a point $x_{o} \in K$ that is not covered by some ball $B \in \calB_{n}$. Define the $C^{1}$-mapping $g \colon B_{0} \to \R$ by $g(y) = |\psi_{B}^{-1}(x_{o}) - y|$. Now $\psi_{B}^{-1}(x_{o}) \notin B_{0}$, so $g$ has no singular points in $K \subset B_{0}$, and we may infer that $\dim g(K) = \min\{1,\dim K\}$. Denoting by $r = d(B) \in (0,1)$ the contraction ratio of $\psi_{B}$ and noticing that $\psi_{B}(K) \subset K$, we have
\begin{align*} D_{x_{o}}(K) & \supset \{|x_{o} - z| : z \in \psi_{B}(K)\} = \{|x_{o} - \psi_{B}(y)| : y \in K\}\\
& = \{r|\psi_{B}^{-1}(x_{o}) - y| : y \in K\} = r \cdot g(K). \end{align*} 
This proves that $\dim D_{x_{o}}(K) = \min\{1, \dim K\}$.
\end{proof}

\section{The Case with only rational rotations}

In essence, the case with only rational rotations can be reduced to the case with \textbf{no} rotations. If $\dim K > 1$, this would be literally true, but we are only assuming that $\calH^{1}(K) > 0$, and this causes some minor technical issues. \begin{lemma}\label{typeId} Let $[O,R]$ be a rational type. Then there exists $n \in \N$ such that if $\{\eta_{1},\ldots,\eta_{n}\}$ are similitudes of type $[O,R]$, then $[\eta_{1} \circ \ldots \circ \eta_{n}] = [\Id,\Id]$. 
\end{lemma}

\begin{proof} Write $O(z) = e^{2\pi i \theta}z$, where $\theta = k/m \in [0,1)$. If $R(z) = \bar{z}$, the claim is valid with $n = 2$. Indeed, if $\eta_{1}(x) = r_{1}e^{2\pi i \theta}\bar{z} + w_{1}$ and $\eta_{2}(x) = r_{2}e^{2\pi i\theta}\bar{z} + w_{2}$ for some $r_{1},r_{2} \in (0,1)$ and $w_{1},w_{2} \in \R^{2}$, we have
\begin{displaymath} \eta_{1} \circ \eta_{2}(z) = r_{1}e^{2\pi i\theta}\overline{(r_{2}e^{2\pi i \theta}\bar{z} + w_{2})} + w_{1} = r_{1}r_{2}z + (r_{1}e^{2\pi i\theta}\bar{w_{2}} + w_{1}), \end{displaymath} 
which means that $[\eta_{1} \circ \eta_{2}] = [\Id,\Id]$. If $R = \Id$, then any composition of the form $\eta_{1} \circ \ldots \circ \eta_{n}$ with $[\eta_{j}] = [O,R]$, $1 \leq j \leq m$, has type $[\Id,\Id]$, and the claim holds with $n = m$.
\end{proof}

The following lemma bears close resemblance to Proposition 6 in \cite{PS} but does not seem to follow from it directly:
\begin{lemma}\label{mainL2} Let $K \subset \R^{2}$ be a self-similar set in $\R^{2}$ generated by similitudes with only rational types. Then, for any  $\varepsilon > 0$, there exists a self-similar set $K_{\varepsilon} \subset K$ satisfying the very strong separation condition, generated by similitudes containing no rotations or reflections, and with $\dim K_{\varepsilon} > \dim K - \varepsilon$.
\end{lemma}

\begin{proof}  By Lemma \ref{mainL1}, we may assume that $K$ already satisfies the very strong separation condition. Let $\{\psi_{1},\ldots,\psi_{q}\}$ be the system of similitudes generating $K$. The rationality assumption implies that there is a \textbf{finite} family of types $\mathcal{T}$ such that every composition of the form $\psi_{i_{1}} \circ \ldots \circ \psi_{i_{k}}$  with $k \in \N$  and $(i_{1},\ldots,i_{k}) \in \{1,\ldots,q\}^{k}$ has one of the types in  $\mathcal{T}$. For every type in $[O,R] \in \mathcal{T}$, fix some such composition $\psi_{[O,R]}$ with $[\psi_{[O,R]}] = [O,R]$. If $[O,R] \in \mathcal{T}$, the previous lemma provides a number $n_{[O,R]} \in \N$ with the following property: if $[\psi] = [O,R]$, then the $n_{[O,R]}$-fold composition $\psi_{[O,R]} \circ \ldots \circ \psi_{[O,R]} \circ \psi$ has type $[\Id,\Id]$. We have proven
\begin{itemize}
\item[($\star$)] For every type $[O,R] \in \mathcal{T}$ there corresponds a number $k[O,R] \in \N$ such that if $\psi$ is a similitude with $[\psi] = [O,R]$, then $[\psi_{i_{1}} \circ \ldots \circ \psi_{i_{k[O,R]}} \circ \psi] = [\Id,\Id]$ for some numbers $i_{j} \in \{1,\ldots,q\}$, $1 \leq j \leq k[O,R]$. Since $\card \mathcal{T} < \infty$, we have
\begin{displaymath} C := \max\{k[O,R] : [O,R] \in \mathcal{T}\} < \infty. \end{displaymath} 
\end{itemize}

Now we start buiding $K_{\varepsilon}$. If $\psi$ is a similitude on $\R^{2}$, write $r_{\psi} \in (0,1)$ for the contraction ratio of $\psi$. The similitude $\psi$ is called \emph{good}, if $[\psi] = [\Id,\Id]$, and \emph{bad} otherwise. Our goal is to construct a family $\calG$ of good similitudes of the form $\psi_{i_{1}} \circ \ldots \circ \psi_{i_{k}}$, with $(i_{1},\ldots,i_{k}) \in \{1,\ldots,q\}^{k}$ and $k \in \N$, satisfying
\begin{equation}\label{goodG} \sum_{\psi \in \calG} r_{\psi}^{1 - \varepsilon} > 1.  \end{equation}
The self-similar set $K_{\varepsilon}$ generated by the similitudes in $\calG$ is contained in $K$ and has $\dim K_{\varepsilon} > \dim K - \varepsilon$ by strong separation. We will find $\calG$ through the following iterative procedure. Initially, let $\calG_{1} := \{\psi_{j}  : 1 \leq j \leq q \text{ and } \psi_{j} \text{ is good}\}$, and $\calB_{1} := \{\psi_{j} : 1 \leq j \leq q \text{ and } \psi_{j} \text{ is bad}\}$. Repeat the algorithm below for $j = 1,2,\ldots$.
\begin{itemize}
\item[(ALG)] If $\calB_{j} = \emptyset$, define $\calG_{j + 1} = \emptyset = \calB_{j + 1}$. Otherwise, for every mapping $\psi \in \calB_{j}$, take all compositions of the form $\psi_{i_{1}} \circ \ldots \circ \psi_{i_{k[\psi]}} \circ \psi$, where $k[\psi]$ is the number from ($\star$). Then, define $\calG_{j + 1}$ and $\calB_{j + 1}$ to consist of all such good and bad compositions, respectively.
\end{itemize}
The idea is to stop iterating (ALG) at some finite stage $j = j_{0}$ and conclude that $\calG := \bigcup_{j \leq j_{0}} \calG_{j}$ satisfies \eqref{goodG}. Writing $t = \dim K$, we clearly have the following equation for every $j \in \N$:
\begin{displaymath} \sum_{i = 1}^{j} \sum_{\psi \in \calG_{i}} r_{\psi}^{t} + \sum_{\psi \in \calB_{j}} r_{\psi}^{t} = 1. \end{displaymath} 
Thus, to prove that \eqref{goodG} holds for $\calG = \bigcup_{j \leq j_{0}} \calG_{j}$ for large enough $j_{0} \in \N$, it suffices to demonstrate that
\begin{equation}\label{badB} \sum_{\psi \in \calB_{j}} r_{\psi}^{t} \to 0 \end{equation}
as $j \to \infty$.
To this end, note that
\begin{displaymath}  \sum_{\psi \in \calB_{j}} r_{\psi}^{t} = \sum_{\psi \in \calB_{j - 1}} r_{\psi}^{t} - \sum_{\psi \in \calG_{j}} r_{\psi}^{t}. \end{displaymath}
Then recall that $k[\psi] \leq C$ for every type $[\psi] \in \mathcal{T}$. This means that, on iterating (ALG), every similitude $\psi \in \calB_{j - 1}$ produces at most $q^{C}$ new similitudes in $\calB_{j} \cup \calG_{j}$ out of which at least one is in $\calG_{j}$, by definition of $k[\psi]$. It follows that 
\begin{displaymath} \sum_{\psi \in \calG_{j}} r_{\psi}^{t} \geq c \sum_{\psi \in \calB_{j - 1}} r_{\psi}^{t} \end{displaymath}
for some constant $c \in (0,1)$ depending only on $q$, $C$ and the contraction ratios of the original similitudes $\psi_{i}$. This gives \eqref{badB} and finishes the proof.
\end{proof}

\begin{lemma}\label{directions} Let $B$ be an $\calH^{1}$-measurable subset of the plane with $\calH^{1}(B) > 0$. Then either $D(B)$ contains an interval, or the \emph{direction set}
\begin{displaymath} S(B) := \left\{\frac{x - y}{|x - y|} : x,y \in B, \: x \neq y\right\} \end{displaymath} 
is dense in $S^{1} = \{x \in \R^{2} : |x| = 1\}$.
\end{lemma}

\begin{proof} If $S(B)$ is not dense in $S^{1}$, there exists a line $L$ through the origin and an opening angle $\alpha > 0$ such that the cones
\begin{displaymath} \calC(x,L,\alpha) := \{ y \in \R^{2} : d(y - x,L) \leq \alpha|x - y|\} \end{displaymath} 
never intersect $B \setminus \{x\}$ for any $x \in B$. It now follows from \cite[Lemma 15.13]{Ma2} that $K$ is $1$-rectifiable.  According to a 1948 result of Besicovitch and Miller \cite{BM}, all $1$-rectifiable sets with positive length in the plane possess the 'Steinhaus property': their distance sets contain an interval.
\end{proof}

\begin{remark}\label{directionsRemark} Applying the previous lemma is the only place in the paper where the assumption $\calH^{1}(K) > 0$ is needed: for everything else, we would be happy with $\dim K \geq 1$. Unfortunately, we were not able to determine whether $\dim K \geq 1$ implies that either the direction set $S(K)$ is dense in $S^{1}$, or $\dim D(K) = 1$.
\end{remark}

\begin{definition}\label{K0} The set $K_{0}$ consists of all points of the form $\psi_{i_{1}} \circ \ldots \circ \psi_{i_{k}}(0)$, with $k \in \N$ and $(i_{1},\ldots,i_{k}) \in \{1,\ldots,q\}^{k}$, such that $[\psi_{i_{1}} \circ \ldots \circ \psi_{i_{k}}] = [\Id,\Id]$. The inclusion $K_{0} \subset K$ is not true in general, but it may be assumed in our situation. Indeed, the distance set of $K$ is invariant under translations of $K$, whence we may always take one of the translation vectors to equal zero -- and then $K_{0} \subset K$.
\end{definition}

\begin{proposition} The set $K_{0}$ is dense in $K$.
\end{proposition}

\begin{proof} Let $\varepsilon > 0$. All points of the form $\psi_{i_{1}} \circ \ldots \circ \psi_{i_{k}}(0)$ with $k \in \N$ and $(i_{1},\ldots,i_{k}) \in \{1,\ldots,q\}^{k}$ are definitely dense in $K$. Thus, given $x \in K$, we may choose an arbitrarily large number $k \in \N$ and a sequence $(i_{1},\ldots,i_{k}) \in \{1,\ldots,q\}^{k}$ such that $\psi_{\varepsilon}(0) \in B(x,\varepsilon)$ with $\psi_{\varepsilon} := \psi_{i_{1}} \circ \ldots \circ \psi_{i_{k}}$. Now choose $k$ so large that the contraction ratio $r_{\varepsilon}$ of $\psi_{\varepsilon}$ is no more than $\varepsilon$. As $[\psi_{\varepsilon}]$ is again rational, the $n$-fold composition of $\psi_{\varepsilon}$ with itself, denoted $\psi_{\varepsilon}^{n}$, has type $[\Id,\Id]$ for some $n \in \N$, by Lemma \ref{typeId}. Then
\begin{displaymath} |\psi_{\varepsilon}^{n}(0) - \psi_{\varepsilon}(0)| = |\psi_{\varepsilon}(\psi_{\varepsilon}^{n - 1}(0)) - \psi_{\varepsilon}(0)| = r_{\varepsilon}|\psi_{\varepsilon}^{n - 1}(0)| \leq \varepsilon. \end{displaymath}
This proves that $\psi_{\varepsilon}^{n}(0) \in B(x,2\varepsilon) \cap K_{0}$.
\end{proof}

The denseness of $K_{0}$ in $K$ and our assumption $\calH^{1}(K) > 0$ immediately yield

\begin{cor}[to Lemma \ref{directions}]  Either $D(K)$ contains an interval, or the direction set $S(K_{0})$ is dense in $S^{1}$.
\end{cor}

If $D(K)$ contains an interval, the proof is finished. So, from now on, we will assume that $S(K_{0})$ is dense in $S^{1}$. For the remainder of the proof, fix $\varepsilon > 0$ and use Lemma \ref{mainL2} to locate a self-similar set $K_{\varepsilon} \subset K$ with $\dim K_{\varepsilon} > \dim K - \varepsilon \geq 1 - \varepsilon$, satisfying the very strong separation condition and containing neither rotations nor reflections. The symbols $\calB_{n}$ and $\calB_{n}^{\varepsilon}$ will be used to denote the collections of generation $n$ balls of $K$ and $K_{\varepsilon}$, respectively. Let $\pi_{e} \colon \R^{2} \to \R$ denote the orthogonal projection $\pi_{e}(x) := x \cdot e$ onto the line spanned by the vector $e \in S^{1}$.

\begin{lemma}\label{mainL3} The set $S_{\varepsilon} := \{e \in S^{1} : \dim \pi_{e}(K_{\varepsilon}) > 1 - \varepsilon\}$ is open and dense in $S^{1}$.
\end{lemma}

\begin{proof} A stronger, more general result is \cite[Proposition 9.3]{HS}. The proof here is presented, not only for its simplicity, but also because it contains information, which will be used independently in the sequel. Marstrand's projection theorem guarantees the denseness of $S_{\varepsilon}$, so we may concentrate on proving openness. Since $K_{\varepsilon}$ contains no rotations or reflections, the projections $\pi_{e}(K_{\varepsilon})$ are self-similar sets in $\R$. Hence, if $\dim \pi_{e}(K_{\varepsilon}) > 1 - \varepsilon$ for some $e \in S^{1}$, Lemma \ref{mainL1} indicates that we may find a self-similar subset $K^{e}$ of $\pi_{e}(K_{\varepsilon})$ satisfying the very strong separation condition and with $\dim K^{e} > 1 - \varepsilon$. The first generation balls $\calB_{1}^{e}$ of $K^{e}$ are then disjoint closed intervals of $\R$, and
\begin{displaymath} \sum_{I \in \calB^{e}_{1}} \ell(I)^{1 - \varepsilon} > 1. \end{displaymath}
Moreover, these intervals are $\pi_{e}$-projections of some \emph{good} balls $\calG_{1}^{e} \subset \bigcup_{m \in \N} \calB_{m}^{\varepsilon}$. Since the intervals in $\calB_{1}^{e}$ are disjoint, the balls $\calG_{1}^{e}$ are contained in well-separated tubes orthogonal to the vector $e$. Now the separation of the tubes can be used to infer that there exists $\delta > 0$ with the following property: if $|\xi - e| < \delta$, the projections $\pi_{\xi}(B)$ and $\pi_{\xi}(B')$ are disjoint intervals for distinct balls $B,B' \in \calG_{1}^{e}$. These intervals satisfy
\begin{displaymath} \sum_{B \in \calG_{1}^{e}} \ell(\pi_{\xi}(B))^{1 - \varepsilon} = \sum_{I \in \calB_{1}^{e}} \ell(I)^{1 - \varepsilon} > 1. \end{displaymath} 
This finishes the proof: $\pi_{\xi}(K_{\varepsilon})$ always contains the self-similar set generated by the similitudes taking $\pi_{\xi}(B_{0})$ to $\pi_{\xi}(B)$, for $B \in \calG_{1}^{e}$. By the previous equation, this self-similar subset of $\pi_{\xi}(K_{\varepsilon})$ has dimension strictly greater than $1 - \varepsilon$ as long as it satisfies the (very) strong separation condition, which is true as long as $|\xi - e| < \delta$.
\end{proof} 

The set $S(K_{0})$ is dense, and the set $S_{\varepsilon}$ is open and non-empty. Hence, we may choose and fix a vector 
\begin{displaymath} e \in S(K_{0}) \cap S_{\varepsilon}. \end{displaymath}
By definition of $e \in S(K_{0})$, we may then locate distinct points $x_{o},y_{o} \in K_{0}$ such that $(x_{o} - y_{o})/|x_{o} - y_{o}| = e$. As $e \in S_{\varepsilon}$, we may also find the collection $\calG_{1} := \calG_{1}^{e} \subset \bigcup_{m \in \N} \calB_{m}^{\varepsilon}$ of \emph{good} balls in the same fashion as in the proof of the previous lemma. The crucial features of these balls are the following:
\begin{equation}\label{goodB1} \sum_{B \in \calG_{1}} d(B)^{1 - \varepsilon} > 1, \end{equation}
and
\begin{equation}\label{goodB2} d(\pi_{e}(B),\pi_{e}(B')) \geq c > 0 \end{equation}
for distinct balls $B,B' \in \calG_{1}$. Based on these facts, we will be able to prove that $\dim D_{x_{o}}(K) > 1 - \varepsilon$, where $D_{x}(R) := \{|x - y| : y \in R\}$ for $x \in \R^{2}$ and $R \subset \R^{2}$. Since $x_{o} \in K_{0} \subset K$, this will give $\dim D(K) > 1 - \varepsilon$ and prove Theorem \ref{main}.

Now, let us see what we can do with the information \eqref{goodB2}. Let $\xi \in S^{1}$ be a vector orthogonal to $e$. Since the balls in $\calG_{1}$ are contained in disjoint (closed) tubes in direction $\xi$, we may choose $\alpha > 0$ such that the following holds: if $B,B' \in \calG_{1}$ are distinct balls, then
\begin{equation}\label{form2} B \cap \calC(x,L_{\xi},\alpha) = \emptyset \quad \text{for all} \quad x \in B', \end{equation} 
where $L_{\xi}$ is the line passing through the origin in direction $\xi$, and, as before,
\begin{displaymath} \calC(x,L_{\xi},\alpha) = \{y \in \R^{2} : d(y - x,L_{\xi}) \leq \alpha|x - y|\}. \end{displaymath}
Let us define one final auxiliary self-similar set: this is the set $G$ generated by the similitudes taking $B_{0}$ to the various balls $B \in \calG_{1}$ without rotations or reflections, see Figure 1. Then $G \subset K_{\varepsilon} \subset K$. Let $\calG_{n}$ be the generation $n$ balls associated with $G$; this definition clearly coincides with the earlier one when $n = 1$.
\begin{figure}[h]
\center
\includegraphics[scale = 0.5]{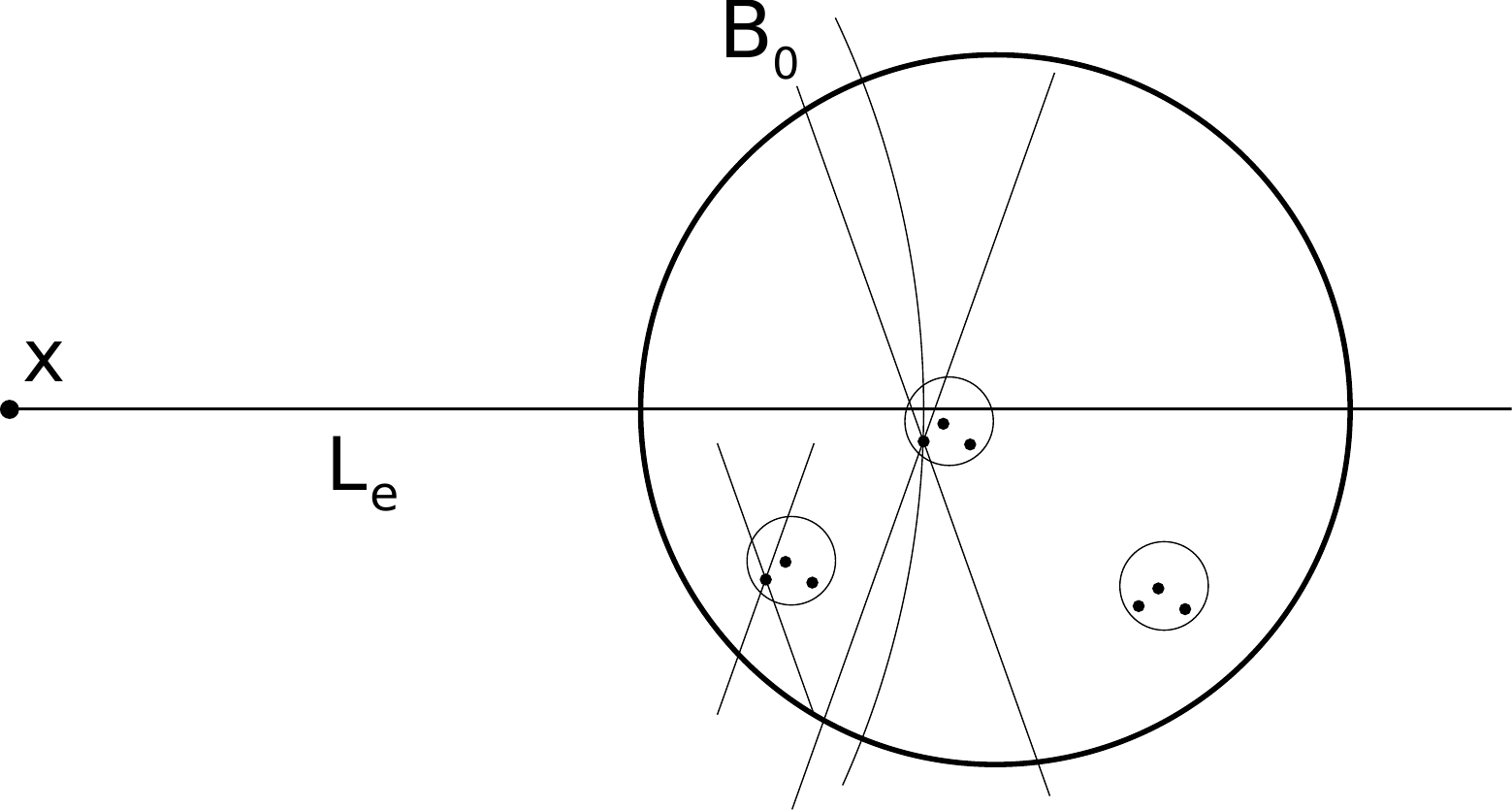}
\caption{The balls in $\calG_{1}$ and $\calG_{2}$ plus some cones of the form $\calC(y,L_{\xi},\alpha)$.}
\end{figure}
\begin{proposition}\label{mainP1} The equation \eqref{form2} is valid for all distinct balls $B,B' \in \calG_{n}$, for any $n \in \N$.
\end{proposition}

\begin{proof} Equation \eqref{form2} settles the case $n = 1$. Suppose that the proposition holds for some $n \in \N$, and choose two distinct balls in $B_{n + 1},B_{n + 1}' \in \calG_{n + 1}$. If $B_{n + 1}$ and $B_{n + 1}'$ are contained in distinct balls of $\calG_{n}$, the claim follows instantly from the induction hypothesis. The other possibility is that $B_{n + 1},B_{n + 1}' \subset B_{n} \in \calG_{n}$. Then, if $\psi$ is the similitude taking $B_{0}$ to $B_{n}$ without rotations or reflections, we have $\psi^{-1}(B_{n + 1}) =: B \in \calG_{1}$ and $\psi^{-1}(B_{n + 1}') =: B' \in \calG_{1}$. Now, for $x \in B_{n + 1}'$ we have $\psi^{-1}(x) \in B'$, whence $B \cap \calC(\psi^{-1}(x),L_{\xi},\alpha) = \emptyset$ by \eqref{form2}. Finally, the identity $\calC(x,L_{\xi},\alpha) = \psi[\calC(\psi^{-1}(x),L_{\xi},\alpha)]$ yields
\begin{displaymath} B_{n + 1} \cap \calC(x,L_{\xi},\alpha) = \psi[B \cap \calC(\psi^{-1}(x),L_{\xi},\alpha)] = \emptyset. \end{displaymath} \end{proof}

Now we are well equipped to study distances. Let $L_{e}$ be the line passing through the origin in the direction of the vector $e$. Then there is a constant $r_{\alpha} > 0$ such that if $x \in L_{e}$ and $|x| \geq r_{\alpha}$, we have
\begin{equation}\label{form3} S(x,|x - y|) \cap B_{0} \subset \calC(y,L_{\xi},\alpha) \quad \text{for all} \quad y \in B_{0}. \end{equation}
This is clear: if $x \in L_{e}$ and $|x|$ is large enough, and if $S = S(x,r)$ is a circle passing through $B_{0}$, then $B_{0} \cap S$ is very close to $B_{0} \cap (L_{\xi} + y)$ for any $y \in B_{0} \cap S$. As a final intermediate step in proving that $\dim D_{x_{o}}(K) > 1 - \varepsilon$, we have
\begin{lemma}\label{lemma10} Suppose that $x \in L_{e}$ with $|x| \geq r_{\alpha}$. Then $\dim D_{x}(G) \geq 1 - \varepsilon$. 
\end{lemma}
We will demonstrate that, under the assumptions of Lemma \ref{lemma10}, the set $D_{x}(G)$ is simply a 'generalized Cantor set', whose dimension is relatively easy to evaluate. To make this precise, we need the following proposition from \cite{MM}:

\begin{proposition}[M. Martin and P. Mattila, simplified]\label{mm} Assume that $p \in \N$, and for every multi-index $(i_{1},\ldots,i_{n}) \in \{1,\ldots,p\}^{n}$ we are given a closed interval $I_{i_{1},\ldots,i_{n}} \subset \R$. Of these intervals, we assume the following:
\begin{itemize}
\item[(i)] $I_{i_{1},\ldots,I_{n},j} \subset I_{i_{1},\ldots,i_{n}}$ for all $j \in \{1,\ldots,p\}$;
\item[(ii)] $I_{i_{1},\ldots,i_{n},i} \cap I_{i_{1},\ldots,i_{n},j}$ for $i \neq j$;
\item[(iii)] $\max\{\ell(I_{i_{1},\ldots,i_{n}}) : (i_{1},\ldots,i_{n}) \in \{1,\ldots,p\}^{n}\} \to 0$ as $n \to \infty$; 
\item[(iv)]
\begin{displaymath} \sum_{j = 1}^{p} \ell(I_{i_{1},\ldots,i_{n},j})^{s} = \ell(I_{i_{1},\ldots,i_{n}})^{s}. \end{displaymath} 
\end{itemize}
Then
\begin{displaymath} 0 < \calH^{s}\left(\bigcap_{n = 0}^{\infty}\bigcup_{(i_{1},\ldots,i_{n})} I_{i_{1},\ldots,i_{n}}\right) < \infty. \end{displaymath}
\end{proposition}
\begin{proof}[Proof of Lemma \ref{lemma10}] The set $D_{x}(G)$ can easily be expressed in a form covered by the proposition. Let $p = \card \calG_{1} \in \N$, and write $\calG_{1} = \{B_{1},\ldots,B_{p}\}$. Then, define $I_{j} = d_{x}(B_{j}) = \{|x - y| : y \in B_{j}\}$. Each ball $B_{j} \in \calG_{1}$ again contains exactly $p$ balls $B_{j1},\ldots,B_{jp} \in \calG_{2}$, and we define $I_{ji} = d_{x}(B_{ji})$ for $1 \leq i \leq p$. It is clear how to continue this process. With the intervals $I_{i_{1},\ldots,i_{n}}$ so defined, let us start verifying the conditions (i)--(iv). There is nothing to prove in (i) and (iii). Also, (iv) is easy, since the balls $B^{1},\ldots,B^{p} \in \calG_{n + 1}$ inside any ball $B^{0} \in \calG_{n}$ satisfy
\begin{displaymath} \sum_{j = 1}^{p} d(B^{j})^{s} = d(B^{0})^{s} \end{displaymath}
for some $s > 1 - \varepsilon$, and $\ell(d_{x}(B)) = d(B)$ for any ball $B \subset \R^{2}$. To establish (ii), we use Proposition \ref{mainP1} and \eqref{form3}. Assume that some intervals $I$ and $I'$ of generation $n \in \N$ overlap, and fix $t \in I \cap I'$. Then $I = d_{x}(B)$ and $I' = d_{x}(B')$ for some balls $B,B' \in \calG_{n}$, and we may find points $y \in B$ and $y' \in B'$ such that $|x - y| = t = |x - y'|$. Then $y \in S(x,|x - y'|) \cap B_{0} \subset \calC(y',L_{\xi},\alpha)$ by \eqref{form3}. But this means that $y \in B \cap \calC(y',L_{\xi},\alpha)$, which is only possible if $B = B'$, by Proposition \ref{mainP1}. Thus $I$ and $I'$ are the same generation $n$ intervals. We have now verified all the conditions of Proposition \ref{mm}, whence
\begin{displaymath} 0 < \calH^{s}(D_{x}(G)) = \calH^{s}\left(\bigcap_{n = 0}^{\infty}\bigcup_{(i_{1},\ldots,i_{n})} I_{i_{1},\ldots,i_{n}}\right) < \infty \end{displaymath}
for some $s > 1 - \varepsilon$. 
\end{proof}

\begin{proof}[Completion of the proof of Theorem \ref{main}] We use all the notation introduced above. Since $G \subset K$, we have now proven that $\dim D_{x}(K) > 1 - \varepsilon$, if $x \in L_{e}$ has norm at least $r_{\alpha}$. We will now demonstrate that $D_{x_{o}}(K)$ contains a scaled copy of $D_{x}(K)$ for some such $x \in L_{e}$. By definition of $y_{o} \in K_{0}$, there exists a similitude $\psi_{y_{o}} := \psi_{i_{1}} \circ \ldots \circ \psi_{i_{k}}$, $1 \leq i_{j} \leq q$, with type $[\psi_{y_{o}}] = [\Id,\Id]$ taking the origin to $y_{o}$. The similitude $\psi_{j}$ mapping the origin to itself (recall the discussion in Definition \ref{K0}) has rational type, so the $m$-fold composition $\psi_{j}^{m}$ of $\psi_{j}$ with itself also has type $[\Id,\Id]$ for some $m \in \N$, by Lemma \ref{typeId}. Let $\psi^{n}$ denote the similitude obtained by composing $\psi_{y_{o}}$ with $n$ copies of $\psi_{j}^{m}$, that is, $\psi^{n} = \psi_{y_{o}} \circ \psi_{j}^{m} \circ \ldots \circ \psi_{j}^{m}$. Then $[\psi^{n}] = [\Id,\Id]$ and $\psi^{n}(0) = y_{o}$ for all $n \in \N$. This means that $\psi^{n}$ has the form $\psi^{n}(x) = \rho_{n}x + y_{o}$ for some constant $\rho_{n}$,  and $\rho_{n} \to 0$ as $n \to \infty$. The point $x_{o}$ lies on the line passing through $y_{o}$ in direction $e$. Hence, $x_{n} := (\psi^{n})^{-1}(x_{o}) \in L_{e}$, and $|x_{n}| = \rho_{n}^{-1}|x_{o} - y_{o}| \to \infty$, as $n \to \infty$. Writing $K_{n} := \psi^{n}(K) \subset K$, we have
\begin{align*} D_{x_{o}}(K) \supset D_{x_{o}}(K_{n}) & = D_{x_{o}}(\psi^{n}(K)) = \{|x_{o} - \psi^{n}(x)| : x \in K\}\\
& = \{\rho_{n}|x_{n} - x| : x \in K\} = \rho_{n} \cdot D_{x_{n}}(K).  \end{align*} 
Since $\dim D_{x_{n}}(K) > 1 - \varepsilon$ for $n \in \N$ large enough, the previous inclusions show that also $\dim D_{x_{o}}(K) > 1 - \varepsilon$. The proof of Theorem \ref{main} is complete.
\end{proof}

\section{Open problems}

\begin{itemize}
\item[(i)] Does $\dim K > 1$ imply positive length for $D(K)$? 
\item[(ii)] Could the hypothesis $\calH^{1}(K) > 0$ be weakened to $\dim K \geq 1$? This would only require improving Lemma \ref{directions}, see Remark \ref{directionsRemark}.

\end{itemize}


\begin{thebibliography}{SSS94}
\bibitem{BM} \textsc{A.S. Besicovitch and D. S. Miller}: \emph{On the Set of Distances Between the Points of Carathéodory Linearly Measurable Plane Point Sets}, Proc. London Math. Soc. \textbf{50}, Ser. 2 (1948), pp. 305--316
\bibitem{Bo} \textsc{J. Bourgain}: \emph{Hausdorff dimension and distance sets}, Israel J. Math \textbf{87} (1994), pp. 193--201
\bibitem{Bo2} \textsc{J. Bourgain}: \emph{On the Erd{\H o}s-Volkmann and Katz-Tao ring conjectures}, Geom. Funct. Anal. \textbf{13} (2003), pp. 334--365
\bibitem{E} \textsc{P. Erd{\H o}s}: \emph{On sets of distances of $n$ points}, Amer. Math. Monthly \textbf{53} (1946), pp. 248-250.
\bibitem{Er} \textsc{M. B. Erdo{\u g}an}: \emph{On Falconer's distance set conjecture}, Rev. Mat. Iberoamericana \textbf{22}(2) (2006), pp. 649--662
\bibitem{Fa} \textsc{K. Falconer}: \emph{On the Hausdorff dimensions of distance sets}, Mathematika \textbf{32} (1985), pp. 206--212
\bibitem{GK} \textsc{L. Guth and N.H. Katz}: \emph{On the Erd{\H o}s distinct distance problem in the plane}, arXiv:1011.4105v3
\bibitem{HS} \textsc{M. Hochman and P. Shmerkin}: \emph{Local entropy averages and projections of fractal measures}, Ann. Math. (to appear)
\bibitem{HI} \textsc{S. Hofmann and A. Iosevich}: \emph{Circular averages and Falconer/Erd{\H o}s distance conjecture in the plane for random metrics}, Proc. Amer. Math. Soc \textbf{133} (2005), pp. 133--143
\bibitem{Hu} \textsc{J. Hutchinson}: \emph{Fractals and self-similarity}, Indiana Univ. Math. J. \textbf{30} (1981), pp. 713--747
\bibitem{IL} \textsc{A. Iosevich and I. {\L}aba}: \emph{K-distance sets, Falconer's conjecture and discrete analogs}, Integers: Electron. J. Number Theory \textbf{5}(2) (2005)
\bibitem{IL2} \textsc{A. Iosevich and I. {\L}aba}: \emph{Distance sets of well-distributed planar point sets}, Discr. Comp. Geom. \textbf{31} (2004), pp. 243--250
\bibitem{IMT} \textsc{A. Iosevich, M. Mourgoglou and K. Taylor}: \emph{On the Mattila-Sjolin theorem for distance sets}, arXiv:1110.6805
\bibitem{KT} \textsc{N. H. Katz and T. Tao}: \emph{Some connections between Falconer's distance set conjecture and sets of Furstenburg type}, New York J. Math. \textbf{7} (2001), pp. 149--187
\bibitem{LG} \textsc{I. {\L}aba and S. Konyagin}: \emph{Distance sets of well-distributed planar sets for polygonal norms}, Israel J. Math. \textbf{152} (2006), pp. 157--179
\bibitem{MM} \textsc{M. Martin and P. Mattila}: \emph{$k$-dimensional regularity classifications for $s$-fractals}, Trans. Amer. Math. Soc. \textbf{305}, No. 1 (1988), pp. 293--315
\bibitem{Ma1} \textsc{P. Mattila}: \emph{Spherical averages of Fourier transforms and measures with finite energy: dimension of intersections and distance sets}, Mathematika \textbf{34} (1987), pp. 207--228
\bibitem{Ma2} \textsc{P. Mattila}: \emph{Geometry of Sets and Measures in Euclidean Spaces}, Cambridge University Press, 1995
\bibitem{MS} \textsc{P. Mattila and P. Sj\"olin}: \emph{Regularity of distance measures and sets}: Math. Nachr. \textbf{204} (1999), pp. 157--162
\bibitem{Mi} \textsc{T. Mitsis}: \emph{A note on the distance set problem in the plane}, Proc. Amer. Math. Soc. \textbf{130}(6) (2001), pp. 1669--1672
\bibitem{PSc} \textsc{Y. Peres and W. Schlag}: \emph{Smoothness of projections, Bernoulli convolutions, and the dimension of exceptions}, Duke Math. J. \textbf{102}(2) (2000), pp. 193--251
\bibitem{PS} \textsc{Y. Peres and P. Shmerkin}: \emph{Resonance between Cantor sets}, Ergodic Theory Dynam. Systems \textbf{29}, No. 1 (2009), pp. 201--221
\bibitem{Ru} \textsc{W. Rudin}: \emph{Real and Complex Analysis}, McGraw-Hill, 1986
\bibitem{Sc} \textsc{A. Schief}: \emph{Separation properties for self-similar sets}, Proc. Amer. Math. Soc. \textbf{122} (1994), pp. 111--115
\bibitem{Wo} \textsc{T. Wolff}: \emph{Decay of circular means of Fourier transforms of measures}, Internat. Math. Res. Notices, 1999, pp. 547--567
\end{thebibliography}
\end{document}